\documentclass[12pt,letterpaper]{article}
\usepackage{amsmath,amsfonts,amsthm,amssymb}

\swapnumbers 

\newtheorem{lemma}{Lemma}[section]
\newtheorem{corollary}[lemma]{Corollary}
\newtheorem{theorem}[lemma]{Theorem}

\newtheorem{assumption}[lemma]{Standing Assumption}

\theoremstyle{definition} 
\newtheorem{definition}[lemma]{Definition}
\newtheorem{remark}[lemma]{Remark}

\newcommand{\Nat}{{\mathbb N}}

\newcommand\bul{\noindent$\bullet\ $}

\begin{document}

\title{Type-Decomposition of a Synaptic Algebra}

\author{David J. Foulis{\footnote{Emeritus Professor, Department of
Mathematics and Statistics, University of Massachusetts, Amherst,
MA; Postal Address: 1 Sutton Court, Amherst, MA 01002, USA;
foulis@math.umass.edu.}}\hspace{.05 in} and Sylvia
Pulmannov\'{a}{\footnote{ Mathematical Institute, Slovak Academy of
Sciences, \v Stef\'anikova 49, SK-814 73 Bratislava, Slovakia;
pulmann@mat.savba.sk. The second author was supported by Research and
Development Support Agency under the contract No. APVV-0178-11 and grant
VEGA 2/0059/12.}}}

\date{}

\maketitle

\begin{abstract}

\noindent A synaptic algebra is a generalization of the 
self-adjoint part of a von Neumann algebra. In this article 
we extend to synaptic algebras the type-I/II/III decomposition 
of von Neumann algebras, AW$\sp{\ast}$-algebras, and JW-algebras.
\end{abstract}

\section{Introduction} \label{sc:intro} 

Soon after laying rigorous Hilbert-space based foundations 
for quantum mechanics in his celebrated book \cite{vNqm},  John 
von Neumann wrote in an unpublished letter to Garrett 
Birkhoff,

\begin{quote}
``I would like to make a confession which may seem immoral: I 
do not believe absolutely in Hilbert space any more."
\end{quote}

\noindent As is authoritatively documented in \cite{Redei}, by the 
time this letter was written (1935), von Neumann had begun to 
focus on what is now called a type II$\sb{1}$ factor as the 
appropriate mathematical basis for quantum mechanics. Later  
von Neumann's advocacy of type II$\sb{1}$ factors was supplemented 
by the discovery that type III factors occur naturally in 
relativistic quantum field theory \cite{Haag}.

In this article we are going to study the type I/II/III 
decomposition theory for a so-called \emph{synaptic algebra}, 
which is a proper generalization of the self-adjoint part of 
a von Neumann algebra. We believe that our work casts considerable 
light on just what makes a type I/II/III decomposition work, not 
only in von Neumann algebras, but in many related algebraic 
structures as well. We note that a synaptic algebra can host the 
probability measures that were a main concern of von Neumann 
\cite[\S 2]{Redei} and it can serve as a value algebra for 
quantum-mechanical observables. 

A synaptic algebra (from the Greek \emph{sunaptein}, meaning 
to join together) \cite{FSynap, FPSynap, SymSA, PuNote} 
unites the notions of an order-unit normed space \cite[p. 69]
{Alf}, a special Jordan algebra \cite{McC}, 
a convex effect algebra \cite{GPBB}, and an orthomodular 
lattice \cite{Beran, Kalm}. The generalized Hermitian algebras, 
introduced and studied in \cite{GHAlg1, GHAlg2}, are special 
cases of synaptic algebras, and numerous additional examples can 
be found in the papers cited above. 

The JW-algebras of D. Topping \cite{Top65} are important special 
cases of synaptic algebras and they will motivate much of our 
work in this article. One of the significant ways in which 
synaptic algebras are more general than JW-algebras is that, 
whereas the orthomodular lattice (OML) of projections in a 
JW-algebra is complete \cite[Theorem 4]{Top65}, the OML of 
projections in a synaptic algebra need not be complete.

Our purpose in this article is twofold: \emph{First Project}: 
To show that a synaptic algebra with a complete OML of 
projections has sufficiently many properties in common with a 
JW-algebra to enable Topping's proof of his version of a 
type-I/II/III decomposition theorem \cite[Theorem 13]{Top65}. 
\emph{Second Project}: To show how the type-decomposition theory 
developed in \cite{FPType} applies to a synaptic algebra with a 
projection lattice satisfying the much weaker central 
orthocompleteness condition. For both projects, our main tool 
will be the notion of a \emph{type determining} (TD) subset 
of the projection lattice (Section \ref{sc:TD} below). 
   
\section{Some basic properties of a synaptic algebra} \label{sc:BPSA}

To carry out our two type-decomposition projects, we shall need 
only a portion of the theory of synaptic algebras as developed 
in \cite{FSynap, FPSynap, SymSA, PuNote}, and as a convenience for 
the reader, we devote this and the next two sections to a sketch 
of some of the parts of this theory that we shall require and to 
some of the corresponding notation and nomenclature. We use the 
notation $:=$ for ``equals by definition" and ``iff" abbreviates 
``if and only if."

\begin{assumption}
In this article we assume that $A$ is a synaptic algebra 
with enveloping algebra $R$ {\rm\cite[Definition 1.1]{FSynap}}. 
\end{assumption}

To help fix ideas, the reader might want to keep in mind the case 
in which $R$ is a von Neumann algebra and $A$ is the self-adjoint 
part of $R$. Some of the important properties of $A$ and $R$  are as 
follows: 

\smallskip

\bul $R$ is a real or complex linear associative algebra with unit 
element $1$ and $A$ is a real vector subspace of $R$. To avoid 
trivialities, we shall assume that $0\not=1$.

\smallskip

\bul $A$ is a partially ordered real vector space under $\leq$ 
and $0\leq 1\in A$.

\smallskip

\bul Let $a,b\in A$. Then the product $ab$ as calculated in 
$R$ may or may not belong to $A$.  We write $aCb$ iff $a$ 
and $b$ commute (i.e. $ab=ba$) and we define $C(a):=\{b
\in A:aCb\}$. If $B\subseteq A$, then $C(B):=\bigcap\sb{b
\in B}C(b)$.
 
\smallskip

\bul If $a\in A$, then $0\leq a\sp{2}\in A$. Thus, $A$ is a 
\emph{special Jordan algebra} under the Jordan product 
$a\circ b:=\frac12((a+b)\sp{2}-a\sp{2}-b\sp{2})=\frac12
(ab+ba)\in A$ for all $a,b\in A$. Therefore, if $a,b\in A$, 
then $ab+ba=2(a\circ b)\in A$ and  $aCb\Rightarrow ab=ba=a
\circ b=b\circ a\in A$. Also, if $ab=0$, then $aCb$ and 
$ba=0$. Moreover, $aba=2a\circ(a\circ b)-a\sp{2}\circ b\in A$ 
and if $0\leq b$, then $0\leq aba$.

\smallskip

\bul With the operations and partial order inherited from $A$, 
the set $C(A)$, called the \emph{center} of $A$, is a synaptic 
algebra with unit element $1$. As such, it is a commutative 
associative partially-ordered normed real linear algebra and 
it is its own enveloping algebra. We call $A$ a \emph{commutative} 
synaptic algebra iff $A=C(A)$.

\smallskip

\bul An element $p\in A$ is called a \emph{projection} iff 
$p\sp{2}=p$, and the set of all projections in $A$ is 
denoted by $P$. Under the partial order inherited from $A$, 
$P$ is an orthomodular lattice (OML) \cite{Beran, Kalm} with 
smallest element $0$, largest element $1$, and $p\mapsto p
\sp{\perp}:=1-p$ as the orthocomplementation. 

\smallskip

\bul If $p\in P$, then, with the operations and partial 
order inherited from $A$, $pAp :=\{pap:a\in A\}$ is a synaptic 
algebra with $pRp$ as its enveloping algebra and with $p$ 
as its unit element. The OML of projections in $pAp$ is 
$pAp\cap P=\{q\in P:q=pqp\}=\{q\in P:q=qp\}=\{q\in P:q\leq p\}$.

\smallskip 

An arbitrary cartesian product of synaptic algebras is again 
a synaptic algebra with coordinatewise operations and relations 
and with the cartesian product of the enveloping algebras of 
the factors as its enveloping algebra.

\section{Orthomodular lattices} \label{sc:OMLs}

In this section we review some facts about orthomodular 
lattices (OMLs) that we shall need in our study of the 
orthomodular projection lattice $P$ of the synaptic 
algebra $A$. More details in regard to OMLs can be found  
in \cite{Beran, Kalm}. 

\begin{assumption}
In this section, we assume that $L$ is an OML with smallest 
element $0$, largest element $1$, and $p\mapsto p\sp{\perp}$ 
as its orthocomplementation.
\end{assumption}

Let $p,q\in L$. We say that $q$ \emph{dominates} $p$, or 
equivalently, that $p$ is a \emph{subelement} of $q$ iff $p
\leq q$. If $p\leq q$ and $p\not=q$, we write $p<q$. As usual, 
$p\vee q$ and $p\wedge q$ will denote the supremum (least upper 
bound) and the infimum (greatest lower bound), respectively, of 
$p$ and $q$ in $L$. The two elements $p,q\in L$ are said to be 
\emph{orthogonal}, in symbols $p\perp q$, iff $p\leq q\sp{\perp}$. 

The \emph{$p$-interval} in $L$, defined and denoted by $L[0,p]
:=\{q\in L:q\leq p\}$ is a sublattice of $L$ and it is an OML 
in its own right with $q\mapsto q\sp{\perp\sb{p}}:=p\wedge q
\sp{\perp}$ as its orthocomplementation. If $Q\subseteq L[0,p]$, 
then $Q$ has a supremum in $L$ iff it has a supremum in $L[0,p]$, 
and the two suprema, if they exist, coincide. Likewise for infima, 
provided that $Q$ is not empty. Therefore, if the OML $L$ is 
\emph{complete} (i.e., every subset of $L$ has a supremum and 
an infimum in $L$) then the OML $L[0,p]$ is also complete.  

The elements $p,q\in L$ are called (\emph{Mackey}) \emph{compatible} 
iff there are elements $p\sb{1},q\sb{1},d\in L$ such that 
$p\sb{1}\perp q\sb{1}$, $p\sb{1}\perp d$, $p\sb{2}\perp d$, 
$p=p\sb{1}\vee p\sb{2}$, and $q=q\sb{1}\vee q\sb{2}$. For instance, 
if $p\leq q$, or if $p\perp q$, then $p$ and $q$ are compatible. 
The elements $p$ and $q$ are compatible iff $p$ can be written 
as $p=x\vee y$ with $x\in L[0,q]$ and $y\in L[0,q\sp{\perp}]$. The 
set of all elements in $L$ that are compatible with every element 
in $L$ is called the \emph{center} of $L$. The center of $L$ is a 
sublattice of $L$, closed under orthocomplementation, and as such 
it is a boolean algebra (a complemented distributive lattice). 
Computations in $L$ are facilitated by the fact that, if one of 
the elements $p,q,r\in L$ is compatible with the other two, then 
the distributive relations $p\vee(q\wedge r)=(p\vee q)\wedge
(p\vee r)$ and $p\wedge(q\vee r)=(p\wedge q)\vee(p\wedge r)$ 
hold \cite{OMLNote}.

If $c$ belongs to the center of $L$ and $p\in L$, then $p\wedge c$ 
belongs to the center of $L[0,p]$. If, conversely, for every $p\in L$, 
every element in the center of $L[0,p]$ has the form $p\wedge c$ for 
some $c$ in the center of $L$, then $L$ is said to have the 
\emph{relative center property} \cite{Chev}.  

\begin{remark} \label{rm:PasanEA}
The OML $L$ can be regarded as a lattice \emph{effect algebra} 
\cite{FBeffect, ZRlatEA} by defining the \emph{orthosum} $p\oplus q$ 
for $p,q\in L$ iff $p\perp q$, in which case $p\oplus q:=p
\vee q$. Then the partial order on $L$ coincides with the 
effect-algebra partial order, the orthocomplementation on $L$ 
coincides with the effect-algebra orthosupplementation, and 
the structure of $L$ as an effect algebra determines its 
structure as an OML. In this way the theory of effect algebras 
\cite{FBeffect, COEA, FPType, HandD, GPBB, JPorthocompl, PtPu, 
ZRlatEA} can be applied to $L$. 
\end{remark}  

A family $(p\sb{i})\sb{i\in I}\subseteq L$ in $L$ is said to be 
\emph{pairwise orthogonal} iff, for all $i,j\in I$, $i\not=j
\Rightarrow p\sb{i}\perp p\sb{j}$. Regarding $L$ as an effect 
algebra and using standard effect-algebra terminology (e.g., 
\cite[p. 286]{COEA}), we have the following: A family in $L$ is 
\emph{orthogonal} iff it is pairwise orthogonal, such an 
orthogonal family is \emph{orthosummable} iff it has a supremum 
in $L$, and if the family is orthosummable, then its 
\emph{orthosum} is its supremum. 

If every orthogonal family in an effect algebra is orthosummable, 
then the effect algebra is called \emph{orthocomplete} \cite
{JPorthocompl}. By a theorem of S. Holland \cite{SSH70}, $L$ 
is orthocomplete as an effect algebra iff it is complete as a 
lattice.

The OML $L$ is said to be \emph{modular} iff, for all $p,q,r\in L$, 
$p\leq r\Rightarrow p\vee(q\wedge r)=(p\vee q)\wedge r$; it is 
called \emph{locally modular} \cite[p. 28]{Top65} iff, for every 
nonzero central element $c\in L$, there is a nonzero $p\in L[0,c]$ 
such that $L[0,p]$ is modular.

If $p,q\in L$, $p\vee q=1$, and $p\wedge q=0$, then $p$ and 
$q$ are called \emph{complements} of each other in $L$. For 
instance, $p$ and $p\sp{\perp}$ are complements in $L$. Two 
elements of $L$ that share a common complement are said to be 
\emph{perspective}. If $p$ and $q$ are perspective in the 
OML $L[0,p\vee q]$, then $p$ and $q$ are called \emph{strongly 
perspective}. The transitive closure of the relation of 
perspectivity is an equivalence relation on $L$ called 
\emph{projectivity}; thus $p$ and $q$ are projective iff 
there is a finite sequence $e\sb{1}, e\sb{2},...,e\sb{n}
\in L$ such that $p=e\sb{1}$, $q=e\sb{n}$, and $e\sb{i}$ is 
perspective to $e\sb{i+1}$ for $i=1,2,...,n-1$. 

\begin{lemma} \label{lm:strongpimpliesp}
Let $r\in L$ and $p,q\in L[0,r]$. Then{\rm: (i)} If $p$ and 
$q$ are perspective in $L[0,r]$, then $p$ and $q$ are 
perspective in $L$. {\rm (ii)} If $p$ and $q$ are strongly 
perspective in $L$, then they are perspective in $L$. {\rm
(iii)} If $p$ and $q$ are strongly perspective in $L$, then 
they are strongly perspective in $L[0,r]$.
\end{lemma}

\begin{proof}
(i) Suppose there exists $x\in L[0,r]$ such that $p\vee x=
q\vee x=r$ and $p\wedge x=q\wedge x=0$ and put $y:=x\vee 
r\sp{\perp}$. Then $p\vee y=p\vee(x\vee r\sp{\perp})=
r\vee r\sp{\perp}=1$, and since $p,x\leq r$, $p\wedge y
=p\wedge(x\vee r\sp{\perp})=(p\wedge x)\vee(p\wedge r
\sp{\perp})=0\vee 0=0$. 
Likewise  $q\vee y=1$ and $q\wedge y=0$.

(ii) Part (ii) follows from (i) with $r:=p\vee q$.

(iii) If $p$ and $q$ are strongly perspective in $L$, 
then they are perspective in $L[0,p\vee q]=(L[0,r])[0,
p\vee q]$, whence they are strongly perspective in 
$L[0,r]$. 
\end{proof}

\begin{theorem} \label{th:modularcondition}
The following conditions are mutually equivalent{\rm:}
\begin{enumerate}
\item $L$ is modular.
\item If $p,q\in L$ are perspective, then $p$ and $q$ 
 are strongly perspective.
\item If $p,q\in L$, $p\leq q$, and $p$ is perspective to $q$, 
then $p=q$.
\end{enumerate}
\end{theorem} 

\begin{proof}
That (i) $\Leftrightarrow$ (ii) follows from \cite[Theorem 2]
{SSH64} and the equivalence (i) $\Leftrightarrow$ (iii) follows 
from \cite[Lemma 20]{Top65}. 
\end{proof}

\begin{theorem} \label{th:contingeom}
Suppose that $L$ is both complete and modular. Then{\rm:}
\begin{enumerate}
\item $L$ is a continuous geometry {\rm\cite{vNcg}}.
\item Perspectivity is transitive on $L$, i.e., perspectivity 
 coincides with projectivity.
\item If $p,q\in L$, $p\leq q$, and $p$ and $q$ are projective  
 in $L$, then $p=q$.
\item Any orthogonal family of nonzero elements in $L$ such 
 that any two of the elements in the family are projective  
 is necessarily finite.
\end{enumerate} 
\end{theorem}

\begin{proof}
Part (i) is a classic result of I. Kaplansky \cite{Kapcg},   
(ii) is \cite[Theorem 5.16]{vNcg}, (iii) is \cite[Theorem 4.4]
{vNcg}, and (iv) is \cite[Theorem 3.8]{vNcg}. 
\end{proof}

\section{The orthomodular lattice of projections} \label{sc:OMLP} 

Owing to the fact that $P\subseteq A$, the OML $P$ of projections in 
$A$ acquires several special properties, among which are the 
following:  

\emph{For all $p,q\in P${\rm: (i)} $p\leq q\Leftrightarrow pq=p
\Leftrightarrow qp=p\Leftrightarrow p=qpq\Leftrightarrow p=pqp$. 
{\rm(ii)} If $p\leq q$, then $q-p=q\wedge p\sp{\perp}=qp\sp{\perp}=p
\sp{\perp}q$. {\rm(iii)} If $pCq$, then $p\wedge q=pq=qp$ and $p\vee q
=p+q-pq$. {\rm(iv)} $p\perp q$ iff $p+q\leq 1$ iff $p+q=p\vee q$ 
iff pq=0. {\rm(v)} $p$ and $q$ are compatible iff $pCq$. {\rm(vi)} 
$C(A)=C(P)$ {\rm\cite[p. 242]{FPSynap}}. {\rm(vii)} The center of 
the OML $P$ is $P\cap C(P)=P\cap C(A)$ and it coincides with the 
boolean algebra of projections in the center $C(A)$ of $A$. 
{\rm(viii)} A projection $c\in P$ is \emph{central}, i.e., it 
belongs to the center $P\cap C(A)=P\cap C(P)$ of $P$, iff $P=
P[0,c]+P[0,c\sp{\perp}]:=\{x+y:x\in P[0,c], y\in P[0,c
\sp{\perp}]\}$. {\rm (ix)} If $d\in P\cap C(A)$, then $pd=
p\wedge d$ belongs to the center of $P[0,p]$. {\rm (x)} If 
$c\in P\cap C(A)$, then the center of $P[0,c]$ is $\{cd:d
\in P\cap C(A)\}=(P\cap C(A))[0,c]$. {\rm(xi)} If $P$ is 
complete, then it has the relative center property 
{\rm\cite[Theorem 8.7]{SymSA}}; hence the center of 
$P[0,p]$ is $\{pd:d\in P\cap C(A)\}$. {\rm(xii)} The 
$p$-interval $P[0,p]$ is the OML of projections in the 
synaptic algebra $pAp$.} 

If $p\sb{1},p\sb{2},...,p\sb{n}$ is a finite orthogonal 
sequence in $P$, then we refer to $p\sb{1}+p\sb{2}+\cdots+
p\sb{n}=p\sb{1}\vee p\sb{2}\vee \cdots\vee p\sb{n}$ as an 
\emph{orthogonal sum}.

\begin{definition}
 The family $(p\sb{i})\sb{i\in I}\subseteq P$ is called 
\emph{centrally orthogonal} iff there is an orthogonal 
family $(c\sb{i})\sb{i\in I}\subseteq P\cap C(A)$ of 
projections in the center $C(A)=C(P)$ of $A$ such that 
$p\sb{i}\leq c\sb{i}$ for every $i\in I$. We say that $P$ 
is \emph{centrally orthocomplete} iff every centrally 
orthogonal family in $P$ has a supremum in $P$.
\end{definition}

Clearly, every centrally orthogonal family is orthogonal, 
and if $P$ is complete, then it is centrally orthocomplete.  
If $P$ is centrally orthocomplete, then the center $P\cap C(P) 
=P\cap C(A)$ is a complete boolean algebra; moreover, for 
each $a\in A$, there is a smallest central projection 
$c\in P\cap C(A)$ such that $a=ac$ \cite[Lemma 6.5 and 
Definition 6.6]{SymSA}. 

\begin{definition} \label{df:centcover}
Suppose that $P$ is centrally orthocomplete. For each 
$a\in A$, the smallest central projection $c\in P\cap 
C(A)$ such that $a=ac$ is called the \emph{central 
cover} of $a$ and denoted by $\gamma a$.
\end{definition}

\noindent Thus, if $P$ is centrally orthocomplete, $a\in A$, 
and $c\in P\cap C(A)$, then $a=ac \Leftrightarrow\gamma a
\leq c$. The restriction of the central cover mapping 
$\gamma\colon A\to P\cap C(A)$ to $P$ is order preserving,  
it preserves arbitrary existing suprema in $P$, and if 
$p\in P$ and $c\in P\cap C(A)$, then $\gamma(p\wedge c)=
\gamma p\wedge c$ \cite[Theorem 6.7]{SymSA}. 

We note that, if $P$ is centrally orthocomplete, then a 
family $(p\sb{i})\sb{i\in I}\subseteq P$ is centrally 
orthogonal iff the family $(\gamma p\sb{i})\sb{i\in I}$ of 
central covers is orthogonal, and it follows that $(p\sb{i})
\sb{i\in I}$ is centrally orthogonal iff it is pairwise 
centrally orthogonal in the sense that, for $i,j\in I$, 
$i\not=j$ implies that the pair consisting of $p\sb{i}$ 
and $p\sb{j}$ is centrally orthogonal. The following lemma 
and theorem address the issue of how these notions 
relativize to an interval $P[0,p]$.

\begin{lemma} \label{lm:relativelycentorth}
Let $p\in P$, let $(p\sb{i})\sb{i\in I}$ be a family 
of projections in $P[0,p]$, and suppose that $(p\sb{i})
\sb{i\in I}$ is centrally orthogonal in $P$. Then 
$(p\sb{i})\sb{i\in I}$ is centrally orthogonal in 
$P[0,p]$.
\end{lemma}

\begin{proof}
By hypothesis, there exists a pairwise orthogonal family 
$(c\sb{i})\sb{i\in I}$ of central projections in $P$ such 
that $p\sb{i}\leq c\sb{i}$ for all $i\in I$. Then 
$(pc\sb{i})\sb{i\in I}$ is a pairwise orthogonal family of 
central projections in $P[0,p]$ and $p\sb{i}\leq pc\sb{i}$ 
for all $i\in I$.
\end{proof}

\begin{theorem} \label{th:relativelycentorth}
Suppose that $P$ is centrally orthocomplete, let $p\in P$, 
let $(p\sb{i})\sb{i\in I}$ be a family of projections in 
$P[0,p]$, and suppose that at least one of the following 
conditions holds{\rm:} $P$ is complete or $p\in P\cap C(A)$. 
Then{\rm: (i)} The family $(p\sb{i})\sb{i\in I}$ is centrally 
orthogonal in $P[0,p]$ iff it is centrally orthogonal in $P$. 
{\rm (ii)} $P[0,p]$ is centrally orthocomplete. 
\end{theorem} 

\begin{proof}
Assume the hypotheses. If $P$ is complete, then it has the 
relative center property, so the center of $P[0,p]$ is 
$\{pc:c\in P\cap C(A)\}$, and the same conclusion holds if 
$p\in P\cap C(A)$.

(i) Suppose that $(p\sb{i})\sb{i\in I}$ is centrally orthogonal in 
$P[0,p]$. Then there exists a family $(c\sb{i})\sb{i\in I}$ of 
central projections in $P$ such that $pc\sb{i}\perp pc\sb{j}$ 
for $i,j\in I$ with $i\not=j$ and $p\sb{i}\leq pc\sb{i}$ for 
all $i\in I$. Then, for $i\not=j$, $p\sb{i}\leq pc\sb{i}\leq 
c\sb{j}\sp{\,\perp}c\sb{i}\in P\cap C(A)$ and $p\sb{j}\leq 
pc\sb{j}\leq c\sb{i}\sp{\,\perp}c\sb{j}\in P\cap C(A)$ with 
$c\sb{j}\sp{\,\perp}c\sb{i}\perp c\sb{i}\sp{\,\perp}c\sb{j}$; 
hence $(p\sb{i})\sb{i\in I}$ is pairwise centrally orthogonal 
in $P$. Since $P$ is centrally orthocomplete, $(p\sb{i})
\sb{i\in I}$ is centrally orthogonal in $P$. The converse 
follows from Lemma \ref{lm:relativelycentorth}, and (i) is 
proved. Since $P$ is centrally orthocomplete, (ii) follows 
from (i).
\end{proof}

Let $c\sb{1},c\sb{2},...,c\sb{n}\in P\cap C(A)$ be a finite 
sequence of central projections with $c\sb{i}\perp c\sb{j}$ 
for $i\not=j$ and $c\sb{1}+c\sb{2}+\cdots+c\sb{n}=1$. Then  
$P$ is the (internal) \emph{direct sum} of the OMLs $P[0,c\sb{i}]$, 
in symbols 
\[
P=P[0,c\sb{1}]\oplus P[0,c\sb{2}]\oplus\cdots\oplus P[0,c\sb{n}], 
\]
in the sense that (1) every projection $p\in P$ can be written 
uniquely as an orthogonal sum $p=p\sb{1}+p\sb{2}+\cdots+p\sb{n}
=p\sb{1}\vee p\sb{2}\vee\cdots\vee p\sb{n}$ with $p\sb{i}\in 
P[0,c\sb{i}]$ for $i=1,2,...,n$ and (2) all operations and 
relations for $P$ can be computed ``coordinatewise" in the obvious 
sense. This direct sum decomposition of $P$ is reflected by a 
corresponding direct sum decomposition $A=c\sb{1}A\oplus c\sb{2}
A\oplus\cdots\oplus c\sb{n}A$ of the synaptic algebra $A$ into the 
direct summands $c\sb{i}A=c\sb{i}Ac\sb{i}=Ac\sb{i}$, where again every 
$a\in A$ can be written uniquely as $a=a\sb{1}+a\sb{2}+\cdots+
a\sb{n}$ with $a\sb{i}\in c\sb{i}A$ for $i=1,2,...,n$ and all 
synaptic operations and relations can be computed 
``coordinatewise." In this case, $P$ is isomorphic as an OML to 
the cartesian product $P[0,c\sb{1}]\times P[0,c\sb{2}]\times
\cdots\times P[0,c\sb{n}]$ and $A$ is isomorphic as a synaptic 
algebra to $c\sb{1}A\times c\sb{2}A\times\cdots\times c\sb{n}A$. 
 
We note that, if $c\in P\cap C(A)$, then $P=P[0,c]\oplus 
P[0,c\sp{\perp}]$ and $A=cA\oplus c\sp{\perp}A$. Thus, the 
direct summands of $P$ (respectively, of $A$) are of the 
form $P[0,c]$ (respectively, $cA$) for central projections 
$c\in P\cap C(A)$.

The OML $P$, is called \emph{irreducible}, and the synaptic 
algebra is said to be a \emph{factor}, iff $P\cap C(A)=
\{0,1\}$. Thus $A$ is a factor iff it admits no nontrivial 
direct-sum decomposition. It can be shown that $A$ is a 
factor iff the center $C(A)$ is the set of all real 
multiples of the unit element $1$. 

By regarding $P$ as an effect algebra, we obtain the following.

\begin{theorem} [{\cite[Theorem 6.14]{COEA}}] \label{th;carprod}
Suppose that $P$ is centrally orthocomplete, let $(p\sb{i})
\sb{i\in I}$ be a centrally orthogonal family in $P$ with $p:=
\bigvee\sb{i\in I}p\sb{i}$, and let $X$ be the cartesian product 
$X:=${\huge$\times$}$\sb{i\in I}P[0,p\sb{i}]$ organized into an 
OML with coordinatewise operations and relations. Define the 
mapping $\Phi\colon X \to P[0,p]$ by $\Phi((e\sb{i})\sb{i\in I})
:=\bigvee\sb{i\in I}e\sb{i}$ for every $(e\sb{i})\sb{i\in I}\in X$. 
Then $\Phi$ is an OML-isomorphism of $X$ onto $P[0,p]$ and for 
$q\in P[0,p]$, $\Phi\sp{-1}(q)=(q\wedge\gamma p\sb{i})\sb{i\in I}
=(q\wedge p\sb{i})\sb{i\in I}$.  
\end{theorem}

\section{Symmetries and equivalence of projections}
 
By a \emph{symmetry} in $A$ we mean an element $s\in A$ such 
that $s\sp{2}=1$ \cite{SymSA}. Two projections $p,q\in P$ are 
said to be \emph{exchanged} by a symmetry $s\in A$ iff $sps=q$, 
or equivalently, iff $sqs=p$. We note that $p$ and $q$ 
are exchanged by a symmetry $s\in A$ iff $p\sp{\perp}=1-p$ 
and $q\sp{\perp}=1-q$ are exchanged by $s$.

\begin{theorem} \label{th:tspst}
Let $p,q\in P$. Then{\rm: (i)} If $p$ and $q$ are exchanged by a 
symmetry in $A$, then $p$ and $q$ are strongly perspective in $P$. 
{\rm(ii)} If $p$ and $q$ are strongly perspective in $P$, then 
$p$ and $q$ are perspective in $P$. {\rm(iii)} If $p$ and $q$ 
are perspective in $P$, then there are symmetries $s,t\in A$ 
such that $tspst=q$. 
\end{theorem}

\begin{proof}
Part (i) follows from \cite[Theorem 5.11]{SymSA}, (ii) is a 
consequence of Lemma \ref{lm:strongpimpliesp} (ii), and 
(iii) follows from \cite[Theorem 5.12 (i)]{SymSA}.
\end{proof}

\begin{corollary} \label{co:tspst}
Let $p,q\in P$ with $p\perp q$. Then the following conditions 
are mutually equivalent{\rm: (i)} There are symmetries $s,t\in A$ 
such that $tspst=q$. {\rm(ii)} $p$ and $q$ are exchanged by a 
symmetry in $A$. {\rm(iii)} $p$ and $q$ are strongly perspective 
in $P$. {\rm(iv)} $p$ and $q$ are perspective in $P$.
\end{corollary}

\begin{proof}
That (i) $\Rightarrow$ (ii) follows from \cite[Theorem 5.12 (ii)]
{SymSA} and (ii) $\Rightarrow$ (iii) $\Rightarrow$ (iv) 
$\Rightarrow$ (i) follows from Theorem \ref{th:tspst}.
\end{proof}
 
\begin{lemma}  \label{lm:rAr}
Let $r\in P$, let $p,q\in P[0,r]$. Then $p$ and $q$ are exchanged 
by a symmetry in $A$ iff $p$ and $q$ are exchanged by a symmetry 
in $rAr$. 
\end{lemma} 

\begin{proof}
Let $p,q\in P[0,r]$, let $s$ be a symmetry in $A$ with 
$sps=q$, and put $t:=(p\vee q)s(p\vee q)$. Clearly, $tpt=q$, 
$tqt=p$ and $rt=tr=t$. Also, since $s(p\vee q)s=sps\vee sqs=
q\vee p=p\vee q$, it follows that $t\sp{2}=p\vee q$, whence 
$t\sp{3}=(p\vee q)t=t$. Put $s\sb{1}:=t+r-t\sp{2}$. Then 
$s\sb{1}=rs\sb{1}r\in rAr$ and $s\sb{1}\sp{\,2}=r$, whence 
$s\sb{1}$ is a symmetry in $rAr$. Moreover, as $rp=pr=p$ and 
$t\sp{2}p=pt\sp{2}=p$, we have $s\sb{1}ps\sb{1}=tpt=q$. 
Conversely, by a straightforward calculation, if $p$ and $q$ 
are exchanged by a symmetry $u$ in $rAr$, then $s:=u+1-r$ is 
a symmetry in $A$ that exchanges $p$ and $q$.      
\end{proof}

\begin{theorem} [{Generalized Comparability}] \label{th:vargencomp}
Suppose that $P$ is complete and let $e,f\in P$. Then there 
exists a symmetry $s\in S$ and a central projection $c\in P
\cap C(A)$ such that $secs\leq fc$, $sfc\sp{\perp}s\leq ec
\sp{\perp}$, $se\sp{\perp}c\sp{\perp}s\leq f\sp{\perp}c\sp
{\perp}$, and $sf\sp{\perp}cs\leq e\sp{\perp}c$.  
\end{theorem}

\begin{proof}
Since $P$ is complete, \cite[Theorem 8.6]{SymSA} applies, 
so there exists $c\in P\cap C(A)$ and a symmetry $s\in S$ such that 
$secs\leq fc$ and $sfc\sp{\perp}s\leq ec\sp{\perp}$. Thus, $fc\sp{\perp}
\leq sec\sp{\perp}s$, so $se\sp{\perp}c\sp{\perp}s=s(c\sp{\perp}-
ec\sp{\perp})s=c\sp{\perp}-sec\sp{\perp}s\leq c\sp{\perp}-fc\sp{\perp}
=f\sp{\perp}c\sp{\perp}$. By a similar computation, $sf\sp{\perp}cs
\leq e\sp{\perp}c$.
\end{proof} 

If $x=s\sb{n}s\sb{n-1}\cdots s\sb{1}\in R$ is a finite product 
of symmetries $s\sb{n}, s\sb{n-1},..., s\sb{1}\in A$ then we 
define $x\sp{\ast}:=s\sb{1}s\sb{2}\cdots s\sb{n}\in R$ to be the 
product of the same symmetries, but in the reverse order. We note 
that $xx\sp{\ast}=x\sp{\ast}x=1$ and, for any $a\in A$, $xax
\sp{\ast}\in A$. 

Let $p,q\in P$. Then by definition, $p$ and $q$ are 
\emph{equivalent}, in symbols, $p\sim q$, iff there is a 
finite sequence of projections $e\sb{1},e\sb{2},...,e\sb{n}
\in P$ such that $p=e\sb{1}$, $q=e\sb{n}$, and for each 
$i=1,2,...,n-1$, the projections $e\sb{i}$ and $e\sb{i+1}$ 
are exchanged by a symmetry $s\sb{i}\in A$. Clearly, $p
\sim q$ iff there is a finite product $x\in R$ of symmetries 
in $A$ such that $q=xpx\sp{\ast}$. 

\begin{lemma} \label{lm:sim&proj}
If $p,q\in P$, then $p\sim q$ iff $p$ and $q$ are 
projective in $P$.
\end{lemma}

\begin{proof}
As a consequence of parts (i) and (ii) of Theorem \ref{th:tspst}, 
if $p\sim q$, then $p$ and $q$ are projective, and the converse 
follows from Theorem \ref{th:tspst} (iii).
\end{proof}

\begin{lemma} \label{lam:xpxstar}
Let $p,q\in P$ and let $x$ be a finite product of symmetries in 
$A$. Then{\rm: (i)} For $a\in A$, the mapping $a\mapsto xax\sp
{\ast}$ is a linear, order, and Jordan automorphism of $A$. 
{\rm (ii)} For $p\in P$, $p\sim xpx\sp{\ast}$ and the mapping 
$p\mapsto xpx\sp{\ast}$ is an OML-automorphism of $P$. {\rm (iii)} 
If $p\in P$, then for $r\in P[0,p]$, the mapping $r\mapsto xrx
\sp{\ast}$ is an OML-isomorphism of $P[0,p]$ onto $P[0,xpx\sp{\ast}]$. 
{\rm (iv)} If $p,q\in P$ and $p\sim q$, then $P[0,p]$ is isomorphic 
as an OML to $P[0,q]$. 
\end{lemma}

\begin{proof}
Parts (i) and (ii) follow from \cite[Theorem 5.3 (i) and (ii)]
{SymSA}. Parts (iii) and (iv) follow from (ii). 
\end{proof}

The projections $p$ and $q$ are said to be \emph{related} 
iff there are nonzero subprojections $0\not=p\sb{1}\leq p$ 
and $0\not=q\sb{1}\leq q$ such that $p\sb{1}\sim q\sb{1}$; 
otherwise they are \emph{unrelated}. If there exists a 
projection $q\sb{1}\leq q$ such that $p\sim q\sb{1}$, 
we say that $p$ is \emph{subequivalent} to $q$, in symbols, 
$p\preceq q$. A projection $h\in P$ is called \emph{invariant} 
iff it is unrelated to its orthocomplement $h\sp{\perp}$.  

\begin{lemma} \label{lm:simandcentcov}
Suppose that $P$ is centrally orthocomplete and let $p,q,h\in P$. 
Then{\rm: (i)} If $p\sim q$, then $\gamma p=\gamma q$. {\rm(ii)} 
If $p\preceq q$, then $\gamma p\leq\gamma q$. {\rm(iii)} 
$h$ is invariant iff it is central. {\rm(iv)} If $P$ is complete, 
then $\gamma p=\bigvee\{q\in P:q\preceq p\}$. {\rm (v)} If 
$\gamma p\perp\gamma q$, then $p$ and $q$ are unrelated. 
{\rm (vi)} If $P$ is complete, then $p$ and $q$ are unrelated 
iff $\gamma p\perp\gamma q$.
\end{lemma}

\begin{proof}
Assume that $P$ is centrally orthocomplete, so the central 
cover mapping $\gamma$ exists. To prove (i), it will be sufficient 
to prove that, if $p$ and $q$ are exchanged by a symmetry $s\in A$, 
then $\gamma p=\gamma q$. So assume that $sps=q$. Since $\gamma p
\in C(A)$, it follows that $q\gamma p=sps\gamma p=sp(\gamma p)s=
sps=q$, whence $\gamma q\leq\gamma p$. Likewise, $\gamma p
\leq\gamma q$, and (i) is proved. Part (ii) is an immediate 
consequence of (i) and the fact that, for $e,f\in P$, $e\leq f
\Rightarrow\gamma e\leq\gamma f$. Part (iii) follows from 
\cite[Theorem 7.5]{SymSA}, (iv) is a consequence of \cite[Theorem 
7.7]{SymSA}, (v) follows from \cite[Corollary 7.6]{SymSA}, and 
(vi) follows from \cite[Corollary 7.8]{SymSA}. 
\end{proof}

We denote the set of natural numbers by $\Nat:=\{1,2,3,...\}$.

\begin{lemma} [{Cf. \cite[Lemma 21]{Top65}}]   \label{le:ikexch} 
If $e\sb{1}, e\sb{2}, e\sb{3}, ...$ is an infinite orthogonal 
sequence of projections and if $e\sb{i}$ and $e\sb{i+1}$ are exchanged 
by a symmetry $s\sb{i}\in A$ for all $i\in\Nat$, then for all $i,j
\in\Nat$, the projections $e\sb{i}$ and $e\sb{j}$ are exchanged 
by a symmetry in $A$.
\end{lemma}

\begin{proof} It will be sufficient to prove by induction on $n
\in\Nat$ that, if $i\in\Nat$ and $i\leq n$, then $e\sb{i}$ and 
$e\sb{n}$ are exchanged by a symmetry in $A$. For $n=1$, this 
is obvious. Assume that it is true for $n$ and suppose that 
$i\in\Nat$ with $i\leq n+1$. We have to prove that $e\sb{i}$ and 
$e\sb{n+1}$ are exchanged by a symmetry in $A$. Obviously, we can 
assume that $i\leq n$, whence by the induction hypothesis, there 
is a symmetry $s\in A$ such that $se\sb{i}s=e\sb{n}$. But $s
\sb{n}e\sb{n}s\sb{n}=e\sb{n+1}$, so $s\sb{n}se\sb{i}ss\sb{n}=
e\sb{n+1}$. Since $e\sb{i} \perp e\sb{n+1}$, we infer from 
Corollary \ref{co:tspst} that $e\sb{i}$ and $e\sb{n+1}$ are 
exchanged by a symmetry in $A$.     
\end{proof}

\begin{lemma} \label{lm:tsfst<f}
Suppose $s$ and $t$ are symmetries in $A$, $f\in P$, and 
$tsfst<f$. Define $f\sb{1}:=f$ and $f\sb{n}:=(ts)\sp{n-1}
f(st)\sp{n-1}$ for $2\leq n\in\Nat$. Then{\rm:} \linebreak  
{\rm (i)} $f=f\sb{1}>f\sb{2}>f\sb{3}>\cdots$. {\rm(ii)} 
The sequence $e\sb{1},e\sb{2},e\sb{3},...\in P[0,f]$ 
defined by $e\sb{n}:=f\sb{n}-f\sb{n+1}>0$ for $n\in\Nat$ 
is orthogonal and for all $i,j\in\Nat$, $e\sb{i}$ and 
$e\sb{j}$ are exchanged by a symmetry in $A$. 
\end{lemma}

\begin{proof} For $n\in\Nat$, we have $f\sb{n+1}=tsf\sb{n}st$. 

(i) We prove by induction on $n\in\Nat$ that $f\sb{n+1}<f\sb{n}$. 
For $n=1$, we have $f\sb{2}=tsfst<f=f\sb{1}$. Assume by the 
induction hypothesis that $n>1$ and $f\sb{n}<f\sb{n-1}$. Then 
$f\sb{n+1}=tsf\sb{n}st<tsf\sb{n-1}st=f\sb{n}$.

(ii) Suppose $i,j\in\Nat$ with $i<j$. Then $f\sb{j}\leq f
\sb{i+1}$ and $f\sb{j-1}\leq f\sb{i}$, whence $e\sb{i}+e
\sb{j}\leq e\sb{i}+e\sb{j}+f\sb{i+1}-f\sb{j}=f\sb{i}-f\sb{j+1}
\leq f\sb{i}\leq 1$, and it follows that $e\sb{i}\perp e\sb{j}$. 
Also, for all $n\in\Nat$, $tse\sb{n}st=ts(f\sb{n}-f\sb{n+1})st=
tsf\sb{n}st-tsf\sb{n+1}st=f\sb{n+1}-f\sb{n+2}=e\sb{n+1}$, and since 
$e\sb{n}\perp e\sb{n+1}$, Corollary \ref{co:tspst} implies that 
$e\sb{n}$ and $e\sb{n+1}$ are exchanged by a symmetry. Therefore, 
by Lemma \ref{le:ikexch}, $e\sb{i}$ and $e\sb{j}$ are exchanged 
by a symmetry for all $i,j\in\Nat$.
\end{proof}

\section{Type-determining sets, orthodensity, and faithful 
projections} \label{sc:TD}

Material in this section is adapted from \cite[\S 3, \S 4]{FPType}. 

\begin{assumption} \label{as:CO}
Henceforth in this article, we assume that the OML $P$ is 
centrally orthocomplete. Therefore the center $P\cap C(A)$ is a 
complete boolean algebra and the central cover mapping 
$\gamma\colon A\to P\cap C(A)$ exists.
\end{assumption}

\begin{definition}   
Let $Q\subseteq P$. Then:
\begin{enumerate}
\item [(1)]  The set of all suprema of centrally orthogonal families of 
 projections in $Q$ is denoted by $[Q]$. We understand that $[\emptyset]
 =\{0\}$.
\item [(2)] $Q\sp{\gamma}:=\{q\wedge c:q\in Q, c\in P\cap C(A)\}$.
\item [(3)] $Q\sp{\downarrow} :=\bigcup\sb{q\in Q}P[0,q]$. If 
 $Q\sp{\downarrow}\subseteq Q\not=\emptyset$, then $Q$ is called 
 an \emph{order ideal}.
\item [(4)] $Q$ is \emph{type determining} (TD) iff $[Q]\subseteq Q$ 
 and $Q\sp{\gamma}\subseteq Q$. 
\item [(5)]$Q$ is \emph{strongly type determining} (STD) iff $[Q]
 \subseteq Q$ and $Q\sp{\downarrow}\subseteq Q$.
\item [(6)] $Q$ is \emph{projective} iff for all $q\in Q$, if 
 $q$ is projective to $p\in P$, then $p\in Q$.
\item [(7)] $Q$ is \emph{orthodense} in $P$ iff every projection in 
 $P$ is the supremum of an orthogonal family of projections in $Q$. 
\item [(8)] $Q$ is an \emph{OML-ideal} iff $Q$ is an order ideal 
 and $p,q\in Q\Rightarrow p\vee q\in Q$. An OML-ideal is a 
 $p$-\emph{ideal} iff it is projective \cite[p. 75]{Kalm}.  
\end{enumerate}
\end{definition}

We note that $Q\subseteq P$ is TD (respectively, STD) iff $Q=[Q]=Q
\sp{\gamma}$ (respectively, iff $Q=[Q]=Q\sp{\downarrow}$). Clearly, 
STD $\Rightarrow$ TD, and the intersection of TD subsets (respectively, 
STD subsets, projective subsets) of $P$ is again TD (respectively, STD, 
projective). Since $[\emptyset]=\{0\}$, $0$ belongs to every TD set. 

If $p\in P$, then the $p$-interval $P[0,p]$ is both an OML ideal 
and an STD subset of $P$, but it is projective iff $p\in C(A)$. Also, 
the center $P\cap C(P)$ is a projective TD subset of $P$, but it is 
STD iff $P$ is boolean. By \cite[Theorem 4.1]{FPType} $[Q\sp{\gamma}]$ 
is the smallest TD subset of $P$ that contains $Q$, and $[Q\sp\downarrow]$ 
is the smallest STD subset of $P$ that contains $Q$. 

Since two projections are projective iff they are equivalent 
(Lemma \ref{lm:sim&proj}), it follows that $Q\subseteq P$ is 
projective iff, for all $p\in P$, $p\sim q\in Q\Rightarrow p
\in Q$. Clearly, $Q$ is projective iff, for every symmetry 
$s\in A$, we have $sQs\subseteq Q$. 

\begin{lemma} \label{lm:relativeTD}
Let $p\in P$ and suppose that one of the following conditions 
holds{\rm:} $P$ is complete or $p\in P\cap C(A)$.  Then, if 
$Q\subseteq P$ is TD {\rm(}respectively, STD, projective{\rm)}
it follows that $Q\cap P[0,p]$ is TD {\rm(}respectively, 
STD, projective{\rm)} both in $P$ and in the projection lattice 
$P[0,p]$ of $pAp$. 
\end{lemma} 

\begin{proof}
Assume the hypotheses. Then, as in the proof of Theorem 
\ref{th:relativelycentorth}, the center of $P[0,p]$ is 
$\{pc:c\in P\cap C(A)\}$; moreover, by Theorem 
\ref{th:relativelycentorth}, $P[0,p]$ is centrally 
orthocomplete and a family in $Q\cap P[0,p]$ is centrally 
orthogonal in $P[0,p]$ iff it is centrally orthogonal in 
$P$. Consequently, $Q\cap P[0,p]$ is closed under the formation 
of suprema of centrally orthogonal families. If $d$ belongs 
to the center of $P[0,p]$, then $d=pc$ for some $c\in P
\cap C(A)$, whence, for any $q\in Q\cap P[0,p]$, we have 
$qd=qpc=qc\in Q\cap P[0,p]$, 
and therefore $Q\cap P[0,p]$ is TD in $P[0,p]$. If $Q$ is STD, 
it is clear that $Q\cap P[0,p]$ is STD in $P[0,p]$. Suppose $Q$ 
is projective, let $q\in Q\cap P[0,p]$, and suppose $q$ and a 
projection $r\in P[0,p]$ are exchanged by a symmetry in $pAp$. Then, 
by Lemma \ref{lm:rAr}, $q$ and $r$ are exchanged by a symmetry 
in $A$, whence $r\in Q\cap P[0,p]$. 
\end{proof}
 
\begin{definition} \label{df:OMLtypeclass}
A nonempty class ${\mathcal L}$ of OMLs is called an \emph{OML type 
class} iff the following conditions are satisfied: (1) If $L\in 
{\mathcal L}$ and $c$ belongs to the center of $L$, then $L[0,c] 
\in{\mathcal L}$. (2) ${\mathcal L}$ is closed under the formation 
of arbitrary cartesian products. (3) If $L\sb{1}$ and $L\sb{2}$ are 
isomorphic OMLs and $L\sb{1}\in{\mathcal L}$, then $L\sb{2}\in
{\mathcal L}$. If, in addition to (2) and (3), ${\mathcal L}$ satisfies 
(1$'$) if $L\in{\mathcal L}$, then $p\in L\Rightarrow L[0,p]\in
{\mathcal L}$, then ${\mathcal L}$ is called a \emph{strong OML type 
class}. 
\end{definition}

Some examples of strong OML type classes are the following: The class 
of all boolean algebras, all modular OMLs, all complete OMLs, all 
$\sigma$-complete OMLs, and all atomic OMLs. Obviously, the 
intersection of (strong) OML type classes is again a (strong) OML 
type class. For instance, the class of all complete modular OMLs is 
a strong OML type class. The class of all locally modular OMLs provides 
an example of an OML type class that is not strong; however the class 
of all complete locally modular OMLs is a strong OML type class.

\begin{theorem} \label{th:classtoset}
If ${\mathcal Q}$ is a OML type class {\rm(}respectively, a strong OML 
type class{\rm)}, then $Q:=\{q\in P:P[0,q]\in{\mathcal Q}\}$ is a 
projective TD set {\rm(}respectively, a projective STD set{\rm)}. 
\end{theorem}

\begin{proof}
Assume that ${\mathcal Q}$ is a OML type class and $Q:=\{q\in P:
P[0,q]\in{\mathcal Q}\}$. Suppose that $(q\sb{i})\sb{i\in I}$ is 
a centrally orthogonal family in $Q$. Since $P$ is centrally 
orthocomplete (Assumption \ref{as:CO}), $q:=\bigvee\sb{i\in I}q
\sb{i}$ exists in $P$. For every $i\in I$, $P[0,q\sb{i}]\in{\mathcal 
Q}$, whence $X:=${\huge$\times$}$\sb{i\in I}P[0,q\sb{i}]\in
{\mathcal Q}$. By Theorem \ref{th;carprod}, $X$ is isomorphic  
as an OML to $P[0,q]$, so $P[0,q]\in{\mathcal Q}$, and therefore 
$q\in Q$. Thus, $[Q]\subseteq Q$. 

Let $q\in Q$ and $c\in P\cap C(A)$. Then $P[0,q]\in{\mathcal Q}$ and, 
$q\wedge c=qc$ belongs to the center of $P[0,q]$, whence $P[0,qc]=
(P[0,q])[0,qc]\in{\mathcal Q}$, and so $qc\in Q$. This proves that 
$Q\sp{\gamma}\subseteq Q$, so $Q$ is a TD-set. To prove that $Q$ is 
projective, let $s\in A$ be a symmetry. Then $P[0,sqs]$ is isomorphic 
as an OML to $P[0,q]\in{\mathcal Q}$, whereupon $P[0,sqs]\in{\mathcal Q}$, 
and we have $sqs\in Q$.

To complete the proof, suppose that ${\mathcal Q}$ is a strong 
OML type class, let $q\in Q$ and suppose $p\in P[0,q]$. Then 
$P[0,p]\in{\mathcal Q}$, and it follows that $(P[0,p])[0,q]=
P[0,q]\in{\mathcal Q}$, whence $p\in Q$.   
\end{proof}

If $Q\subseteq P$, we understand that $\gamma(Q):=\{\gamma q:
q\in Q\}$. The following theorem is an adaptation to our 
present context of \cite[Theorem 4.5 and Corollary 4.6]{FPType}. 

\begin{theorem} \label{th:TDEAth4.5}
Let $Q\subseteq P$ be a TD set. Then{\rm: (i)} $Q\cap \gamma(Q)=
Q\cap C(A)\subseteq \gamma(Q)\subseteq P\cap C(A)$. {\rm(ii)} There 
is a unique central projection $c\sb{Q}\in P\cap C(A)$ such that 
$\gamma(Q)=(P\cap C(A))[0,c\sb{Q}]$. {\rm(iii)}There is a unique 
central projection $c\sb{Q\cap C(A)}\in P\cap C(A)$ such that 
$Q\cap\gamma(Q)=Q\cap C(A)=(P\cap C(A))[0,c\sb{Q\cap C(A)}]$. 
{\rm(iv)} Both $\gamma(Q)$ and $Q\cap\gamma(Q)=Q\cap C(A)$ are 
TD subsets of $P$. 
\end{theorem}

\begin{definition} [{Cf. \cite[Definition 4.7]{FPType}}]
Let $Q$ be a TD subset of $P$. Then the central projection 
$c\sb{Q}$ in Theorem \ref{th:TDEAth4.5} (ii) is called the 
\emph{type-cover} of $Q$, and the central projection 
$c\sb{Q\cap C(A)}$ in Theorem \ref{th:TDEAth4.5} (iii) is 
called the \emph{restricted type-cover} of $Q$.  
\end{definition}

\noindent The type cover $c\sb{Q}$ and the restricted type 
cover $c\sb{Q\cap C(A)}$ will play significant roles in Sections 
\ref{sc:fundamental} and \ref{sc:I/II/III} below. 

\begin{lemma} [{\cite[Lemma 4.8]{FPType}}] \label{lm:typecover}
Let $Q\subseteq P$ be a TD set. Then{\rm: (i)} $c\sb{Q}$ is 
the largest projection in $\gamma(Q)$ and every central 
subprojection of $c\sb{Q}$ belongs to $\gamma(Q)$. {\rm(ii)} 
$c\sb{Q\cap C(A)}$ is the largest central projection in $Q$. 
{\rm(iii)} $c\sb{Q\cap C(A)}\leq c\sb{Q}$. {\rm(iv)} The 
smallest central projection $c\in P\cap C(A)$ such that 
$Q\subseteq P[0,c]$ is $c=c\sb{Q}$.  {\rm(v)} The smallest 
central projection $d\in P\cap C(A)$ such that $Q\cap 
C(A)\subseteq P[0,d]$ is $d=c\sb{Q\cap C(A)}$. {\rm(vi)} 
If $c\in P\cap C(A)$, then $c\perp c\sb{Q}$ iff $Q\cap 
P[0,c]=\{0\}$. {\rm(vii)} If $c\in P\cap C(A)$, then 
$c\perp c\sb{Q\cap C(A)}$ iff $Q\cap (P\cap C(A))[0,c]=
\{0\}$.
\end{lemma}

\begin{definition} \label{df:faithful}
A projection $f\in P$ is \emph{faithful} iff $\gamma f=1$.
\end{definition}

\noindent Clearly, $f\in P$ is faithful iff the only 
central projection $c\in P\cap C(A)$ such that $f\leq c$ 
is $c=1$. The next lemma clarifies how faithfulness 
relativizes to a direct summand of $P$. 

\begin{lemma} [{\cite[Lemma 3.5]{FPType}}] \label{lm:relfaith}
Let $c\in P\cap C(A)$ and let $f\in P[0,c]$. Then the 
following conditions are mutually equivalent{\rm: (i)} 
$f$ is faithful in the projection lattice $P[0,c]$ of 
$cA$. {\rm (ii)} $\gamma f=c$. {\rm(iii)} $\gamma(P[0,f])$
is the center of $P[0,c]$. {\rm(iv)} $f$ has a nonzero 
component in every nonzero direct summand of $P[0,c]$, 
i.e., if $0\not=d\in(P\cap C(A))[0,c]$, then $fd\not=0$. 
\end{lemma}

\begin{lemma} \label{lm:orthodense}
Let $Q\subseteq P$ and let $c:=\bigvee\gamma(Q)$. Then 
$Q\subseteq P[0,c]$ and the following conditions are 
mutually equivalent{\rm: (i)} If $p\in P$, $d\in\gamma(Q)$, 
and $pd\not=0$, then $Q\cap P[0,p]\not=\{0\}$. {\rm(ii)} 
If $0\not=p\in P[0,c]$, then $Q\cap P[0,p]\not=\{0\}$. 
{\rm(iii)} $Q$ is orthodense in $P[0,c]$.
\end{lemma}

\begin{proof}
If $q\in Q$, then $q\leq\gamma q\leq c$, so $Q\subseteq P[0,c]$. 
The rest of the lemma follows from \cite[Lemma 5.3]{HandD} by taking 
$\eta=\gamma$.
\end{proof}

\begin{theorem} [{Cf. \cite[Propositions 13 and 16]{Top65}}] 
\label{th:TopProp13}
Suppose that $P$ is complete, $Q\sp{\downarrow}\subseteq Q
\subseteq P$, $Q$ is projective, and $c=\bigvee\{\gamma q:
q\in Q\}$. Then {\rm(i)} $Q\subseteq P[0,c]$. {\rm(ii)} 
$Q$ is orthodense in $P[0,c]$. {\rm(iii)} $c=\bigvee Q$. 
{\rm(iv)} If $Q$ is TD, then $c=c\sb{Q}\in\gamma(Q)$.  
\end{theorem}

\begin{proof}
Assume the hypotheses. Part (i) follows from Lemma \ref
{lm:orthodense}. To prove (ii), it will be sufficient to 
show that part (i) of Lemma \ref{lm:orthodense} holds. Thus, 
suppose that $p\in P$, $q\in Q$, and $p\wedge\gamma q\not=0$. 
Since $P$ is complete, \cite[Corollary 7.8]{SymSA} applies, 
whence $p$ and $q$ are related, i.e., there are nonzero 
subprojections $0\not=p\sb{1}\leq p$ and $0\not=q\sb{1}
\leq q$ such that $p\sb{1}\sim q\sb{1}$. As $q\sb{1}\in Q
\sp{\downarrow}\subseteq Q$ and $Q$ is projective, it 
follows that $p\sb{1}\in Q$, and (ii) is proved. By (ii), 
$c$ is the supremum of an orthogonal family in $Q$, by (i), 
$c$ is an upper bound for $Q$, whence (iii) holds. If 
$Q$ is TD, then by Lemma \ref{lm:typecover} (i), 
$c\sb{Q}$ is the largest projection in $\gamma(Q)$, 
whence $c\sb{Q}=\bigvee\gamma(Q)=c$. 
\end{proof}

\begin{corollary} \label{co:orthodense}
Suppose that $P$ is complete, $Q\subseteq P$, $Q$ is projective, 
and $Q$ is STD. Then $c\sb{Q}=\bigvee Q$, $Q\subseteq P[0,c\sb{Q}]$,  
and $Q$ is orthodense in $P[0,c\sb{Q}]$.
\end{corollary}

\begin{theorem} \label{th:gammaq=gammap}
Let $Q\subseteq P$ be TD and let $0\not=p\in P$. Then the 
following two conditions are equivalent{\rm: (i)} There exists 
$0\not=q\in Q\cap P[0,p]$ such that $\gamma q=\gamma p$. 
{\rm (ii)} For all $d\in P\cap C(A)$, if $pd\not=0$, then 
$Q\cap P[0,pd]\not=\{0\}$. 
\end{theorem}

\begin{proof}
(i) $\Rightarrow$ (ii). Assume (i) and let $d\in P\cap C(A)$ 
with $pd\not=0$.  By (i), there exists $0\not=q\in Q\cap 
P[0,p]$ with $\gamma q=\gamma p$. Put $q\sb{0}:=qd=q\wedge d
\leq p\wedge d=pd$. Then $q\sb{0}\in Q\sp{\gamma}\subseteq Q$, 
and since $d\in P\cap C(A)$, we have $0\not=pd\leq\gamma(pd)=
(\gamma p)d=(\gamma q)d=\gamma(qd)=\gamma q\sb{0}$, whence 
$q\sb{0}\not=0$.

(ii) $\Rightarrow$ (i). Assume (ii), let $(q\sb{i})\sb{i\in I}$ 
be a maximal centrally orthogonal family in $Q\cap P[0,p]$, 
and put $q:=\bigvee\sb{i\in I}q\sb{i}\in[Q]\subseteq Q$. Then 
$q\leq p$, so $\gamma q\leq\gamma p$. Taking $d=1$ in (ii), we 
find that $Q\cap P[0,p]\not=\{0\}$, whence $q\not=0$. If 
$\gamma q=\gamma p$, then we are done, so, we assume that 
$\gamma q<\gamma p$ and this time we put $d=\gamma p-\gamma q
=\gamma p(\gamma q)\sp{\perp}$ in (ii). Then, as $p\leq\gamma p$, 
we have $pd=p(\gamma q)\sp{\perp}$, whence, if $pd=0$, then 
$p\leq\gamma q$, so $\gamma p\leq \gamma q\leq \gamma p$, 
contradicting $\gamma q<\gamma p$. Therefore, $pd\not=0$, and 
it follows from (ii) that there exists $0\not=q\sb{0}\in Q
\cap P[0,pd]$. But then, $q\sb{0}\leq p$ and $q\sb{0}\leq d
\leq(\gamma q)\sp{\perp}\leq(\gamma q\sb{i})\sp{\perp}$ for 
all $i\in I$, contradicting the maximality of $(q\sb{i})
\sb{i\in I}$.
\end{proof}

\begin{corollary} \label{co:cingammaQ}
Let $c\in P\cap C(A)$ and let $Q$ be a TD subset of $P$. Then 
the following two conditions are equivalent{\rm: (i)} $c\in 
\gamma(Q)$. {\rm (ii)} $Q$ has a nonzero intersection with 
every nonzero direct summand of $P[0,c]$. 
\end{corollary} 

\begin{proof}
Assume the hypotheses. If $c=0$, then (i) and (ii) are both 
true, so we assume that $c\not=0$ and put $p:=c$ in 
Theorem \ref{th:gammaq=gammap}. Then conditions (i) and (ii) 
in Theorem \ref{th:gammaq=gammap} are equivalent to conditions 
(i) and (ii) in the corollary.
\end{proof}

\begin{lemma} [{Cf. \cite[Proposition 15]{Top65}}] 
Suppose that $P$ is complete, $Q\subseteq P$, $Q$ is projective, 
$Q$ is STD, $c\in\gamma(Q)$, and $0\not=p\in P[0,c]$. Then  
there exists $0\not=q\in Q\cap P[0,p]$ with $\gamma q=\gamma p$.
\end{lemma}

\begin{proof}
Assume the hypotheses and suppose that $d\in P\cap C(A)$ with 
$pd\not=0$. By Theorem \ref{th:gammaq=gammap}, it will be 
sufficient to prove that $Q\cap P[0,pd]\not=\{0\}$. But, by 
Theorem \ref{th:TDEAth4.5} (ii), $c\leq c\sb{Q}$, whence, if 
$0\not=pd\in P[0,c]$, then $0\not=pd\in P[0,c\sb{Q}]$, and it 
follows from Corollary \ref{co:orthodense} that $pd$ is the 
supremum of an orthogonal family in $Q$. Therefore, since $pd
\not=0$, it follows that $Q\cap P[0,pd]\not=\{0\}$. 
\end{proof}

\section{Abelian, modular, locally modular, and\newline  
complete projections} \label{sc:Abelianetc}

In this section we study some important examples of 
TD and STD subsets of $P$. Many of the results in this 
section are generalizations to a synaptic algebra of 
results due to D. Topping for JW-algebras \cite{Top65}. 
Often the proofs of these results are more or less the 
same as Topping's proofs, but we include these proofs here 
in the interest of a more coherent account. \emph{The 
assumption that $P$ is centrally orthocomplete is 
still in force.}

\begin{definition} \label{df:abelianetc}
Let $p\in P$.
\begin{enumerate}
\item [(1)] $p$ is \emph{abelian} (also called \emph{boolean} 
 \cite[p. 1551]{FPType} iff $P[0,p]$ is a boolean algebra. 
 (We shall regard  $P[0,0]=\{0\}$ as a ``degenerate" boolean 
 algebra, hence $0$ is an abelian projection in $A$.) We 
 denote the set of all abelian projections in $P$ by $B$.
\item [(2)] $p$ is \emph{modular} iff $P[0,p]$ is a modular 
 OML. We denote the set of all modular projections in $P$ 
 by $M$.
\item [(3)] $p$ is \emph{locally modular} iff $P[0,p]$ is 
 a locally modular OML. We denote the set of all locally modular 
 projections in $P$ by $M\sb{0}$.
\item [(4)] $p$ is \emph{complete} iff $P[0,p]$ is a complete 
 OML. We denote the set of all complete projections in $P$ 
 by $T$. 
\end{enumerate}
\end{definition}

\begin{theorem} \label{th:BMTSTD}
{\rm(i)} $B\subseteq M\subseteq M\sb{0}$. {\rm(ii)} The sets 
$B$, $M$, and $T$ are projective STD sets. {\rm(iii)} $M\sb{0}$ 
is a projective TD set. {\rm(iv)}  If $c\in P\cap C(A)$, then 
$c\in M\sb{0}\Leftrightarrow c\in\gamma(M)$.
\end{theorem} 

\begin{proof}
Part (i) is obvious. Since the class ${\mathcal B}$ of all 
boolean OMLs, the class ${\mathcal M}$ of all modular OMLs, 
and the class ${\mathcal T}$ of all complete OMLs are strong 
OML type classes and the class ${\mathcal M}\sb{0}$ of all 
locally modular OMLs is an OML type class, (ii) and (iii)  
follow from Theorem \ref{th:classtoset}, and part (iv) 
follows from Corollary \ref{co:cingammaQ}. 
\end{proof}

\begin{lemma}  \label{lm:AbelianindA}
Let $p,q\in P$. Then{\rm: (i)} $p\in B$ iff $pAp$ is a 
commutative synaptic algebra. {\rm: (ii)} If $p\in B$, then 
$p\wedge q$ is an abelian projection in the synaptic algebra 
$qAq$. {\rm (iii)} If $p\in P[0,q]$, then $p\in B$ iff $p$ 
is an abelian projection in the synaptic algebra $qAq$.
\end{lemma}

\begin{proof}
(i) By \cite[Theorem 4.5]{SymSA}, $pAp$ is a commutative synaptic 
algebra iff its lattice of projections $P[0,p]$ is a boolean 
algebra.  

(ii) Suppose $p\in B$, i.e., $P[0,p]$ is boolean. Then 
$q\wedge p\in P[0,q]$, which is the lattice of projections 
in $qAq$, and $(P[0,q])[0,q\wedge p]=P[0,q\wedge p]$. But,
since $P[0,p]$ is boolean, so is the sublattice $P[0,q\wedge p]
\subseteq P[0,p]$, whence $q\wedge p$ is abelian in $qAq$.  

(iii) If $p$ is abelian in $qAq$, then $(P[0,q])[0,p]=P[0,p]$ 
is boolean, whence $p\in B$. Conversely, suppose that $p
\in P[0,q]$, i.e., $p\leq q$. Then, if $p\in B$, it follows from 
(ii) that $p=p\wedge q$ is abelian in $qAq$. 
\end{proof}

\begin{theorem} [{Cf. \cite[Theorem 11]{Top65}}] \label{th:modfin}
Let $p\in P$ and consider the following two conditions{\rm: (i)} 
Every orthogonal family of nonzero projections in $P[0,p]$, any 
two of which are exchanged by a symmetry in $A$, is necessarily 
finite. {\rm(ii)} $p\in M$. Then {\rm(i)} $\Rightarrow$ {\rm(ii)}, 
and if $p\in T$, then {\rm(ii)} $\Rightarrow$ {\rm(i)}.
\end{theorem}

\begin{proof} To prove that (i) $\Rightarrow$ (ii), it will be 
sufficient to show that if (ii) fails, then (i) fails. So assume 
that $P[0,p]$ is not modular. Then by Theorem \ref
{th:modularcondition}, there exist projections $e,f\in P[0,p]$ 
such that $e<f$ and $e$ is perspective to $f$ in $P[0,p]$. Thus 
by Lemma \ref{lm:strongpimpliesp} (i), $e$ is perspective to $f$ 
in $P$, whence by Theorem \ref{th:tspst} (iii), there are symmetries 
$s,t\in A$ such that $stets=f$. Therefore, $tsfst=e<f$, and by 
Lemma \ref{lm:tsfst<f} (ii), (i) fails. 

Conversely, assume that $P[0,p]$ is complete, that (ii) holds, 
and that $(e\sb{i})\sb{i\in I}$ is an orthogonal family of 
nonzero projections in $P[0,p]$ any two of which are exchanged 
by a symmetry in $A$. Therefore, by Corollary \ref{co:tspst}, 
any two projections in $(e\sb{i})\sb{i\in I}$ are strongly 
perspective in $P$, whence by Lemma \ref{lm:strongpimpliesp}, 
they are strongly perspective, hence perspective, in $P[0,p]$. 
Since the OML $P[0,p]$ is modular and complete, it follows 
from Theorem \ref{th:contingeom} (iv) that $(e\sb{i})\sb{i\in I}$ 
is finite.  
\end{proof}

\begin{lemma} [{Cf. \cite[Lemma 23]{Top65}}] \label{le:1nonmod} 
Suppose that $p\in T$, but $p\notin M$. Then there is a 
projection $e\in P[0,p]$ with the following properties{\rm: 
(i)} $e$ is the supremum of an infinite sequence of nonzero 
projections in $P[0,p]$ any two of which are exchanged by a 
symmetry in $A$. {\rm (ii)} There is a symmetry $s\in A$ with 
$ses\in P[0,p]$ and $ses\perp e$.
\end{lemma}

\begin{proof}
Assume the hypotheses. Then by Theorem \ref{th:modfin}, there is 
an infinite sequence $e\sb{1}, e\sb{2}, e\sb{3},...$ of nonzero 
projections in $P[0,p]$, any two of which are exchanged by a 
symmetry in $A$, whence also by a symmetry in $pAp$ (Lemma 
\ref{lm:rAr}). Putting $e:=\bigvee\sb{n=1}\sp{\infty}e\sb{2n}$, 
we have (i). To prove (ii), we work in the synaptic algebra 
$pAp$ and its complete OML $P[0,p]$ of projections. Let $f:=
\bigvee\sb{n=1}\sp{\infty}e\sb{2n-1}$. Then $e\perp f$, whence 
by \cite[Theorem 5.15]{SymSA} (a weak form of additivity for 
exchangeability by symmetries), there is a symmetry $t\in pAp$ 
such that $tet=f$. By Lemma \ref{lm:rAr} again, there is a 
symmetry $s\in A$ with $ses=f$, and (ii) is proved.
\end{proof}

\begin{theorem} [{Cf. \cite[Theorem 12]{Top65}}] \label{th:supmod} 
{\rm(i)} If $p,q\in M$ and $p\vee q\in T$, then $p\vee q\in M$. 
{\rm(ii)} If $P$ is complete, then $M$ is both a projective 
STD set and a $p$-ideal in $P$.
\end{theorem}

\begin{proof} (i) Assuming the hypothesis of (i), we have to 
prove that $P[0,p\vee q]$ is modular; hence we may drop down to 
the synaptic algebra $(p\vee q)A(p\vee q)$ with complete projection 
lattice $P[0,p\vee q]$. Thus, changing notation, we can (and do) 
assume that $P$ is complete, that $p,q\in M$ with $p\vee q=1$, 
and we have to prove that $P$ is modular. By \cite[Theorem 5.9 (ii)]
{SymSA} (the \emph{symmetry parallelogram law}) $p\sp{\perp}=
1-p=(p\vee q)-p$ is exchanged by a symmetry in $A$ with the modular 
projection $q-(p\wedge q)\leq q$, so $p\sp{\perp}$ is modular.

Now, aiming for a contradiction, we assume that $P$ is not modular. 
Therefore by Lemma \ref{le:1nonmod} (with p=1), there is a projection 
$e\in P$ such that $e$ is the supremum of an infinite sequence of 
nonzero projections in $P$ any two of which are exchanged by a 
symmetry in $A$, and there is a symmetry $t\in A$ with $tet\perp e$.
Applying Theorem \ref{th:vargencomp} to the pair $e,p$, we find that 
there is a symmetry $s\in A$ and a central projection $c\in P
\cap C(A)$ such that $secs\leq pc$ and $se\sp{\perp}c\sp{\perp}s
\leq p\sp{\perp}c\sp{\perp}$. From the latter inequality and the 
fact that $p\sp{\perp}\in M$, we infer that $e\sp{\perp}c
\sp{\perp}\in M$. But $tec\sp{\perp}t=tetc\sp{\perp}\leq 
e\sp{\perp}c\sp{\perp}$, whence $ec\sp{\perp}\in M$. Moreover, 
the pair of modular projections $ec$ and $ec\sp{\perp}$ is  
centrally orthogonal, hence $e=ec+ec\sp{\perp}\in M$, 
contradicting Theorem \ref{th:modfin}.

(ii) Part (ii) follows immediately from (i). 
\end{proof}

\begin{theorem} [{Cf. \cite[Corollary 21]{Top65}}]  \label{th:mod} 
Assume that $P$ is complete and let $p,q\in P$. Then{\rm:}
\begin{enumerate}
\item If $p\in M$ and $p\sim q$, then $q\in M$, there is a 
 projection $r\in M$ such that $p,q\in P[0,r]$, and $p$ is 
 perspective to $q$ in $P[0,r]$.
\item If $p\in M$ $p\sim q$, then $q\in M$ and $p$ and $q$ 
 are perspective in $P$. 
\item On the set $M$, perspectivity is transitive. 
\item If $q\leq p\in M$ and $q\sim p$, then $q=p$ {\rm(}i.e., 
 $p$ is finite {\rm\cite[p.23]{Top65})}.
\item If $p,q\in M$, then $p\sim q$ iff $p$ and  $q$ are  
 exchanged by a symmetry in $A$.
\item If $p,q\in M$, $p\preceq q$, and $q\preceq p$, then 
 $p\sim q$.
\end{enumerate}
\end{theorem}

\begin{proof} (i) Assume $p\in M$ and $p\sim q$. Since $M$ is 
projective, $q\in M$. Also there exist projections $p=e\sb{1},
e\sb{2},...,e\sb{n}=q$ such that $e\sb{i}$ is exchanged by a 
symmetry in $A$ with $e\sb{i+1}$ for $i=1,2...,n-1$. Since $p
\in M$, it follows from Lemma \ref{lam:xpxstar} (iii) that 
$e\sb{1},e\sb{2},...e\sb{n}\in M$, and by Theorem \ref{th:supmod}, 
$r:=e\sb{1}\vee e\sb{2}\vee\cdots\vee e\sb{n}\in M$. By Lemma 
\ref{lm:rAr}, for $i=1,2,...,n-1$, $e\sb{i}$ is exchanged with 
$e\sb{i+1}$ by a symmetry in $rAr$. Since $P$ is complete, so 
is $P[0,r]$; hence, we may apply Theorem \ref{th:contingeom} 
(ii) to $rAr$ and its complete modular projection lattice 
$P[0,r]$ and infer that $p=e\sb{1}$ is perspective to $q=
e\sb{n}$ in $P[0,r]$.

(ii) By (i) and Lemma \ref{lm:strongpimpliesp}, $p$ and $q$ are 
perspective in $P$.

(iii) Suppose that $p,q,r\in M$ with $p$ perspective to $q$ and $q$ 
perspective to $r$ in $P$. Then by Theorem \ref{th:tspst} (iv), 
$p\sim q$ and $q\sim r$, so $p\sim r$, and by (ii), $p$ is 
perspective to $r$ in $P$.

(iv) Assume that $q\leq p\in M$ and $q\sim p$. By (i) there 
exists $r\in M$ such that $q\leq p\in P[0,r]$ and $p$ is 
perspective to $q$ in $P[0,r]$; hence $p=q$ by Theorem 
\ref{th:modularcondition} applied to the modular OML $P[0,r]$.  

(v) Suppose that $p\in M$ $p\sim q$. Then $q\in M$ and applying 
Theorem \ref{th:vargencomp} we infer that there is a symmetry 
$s\in A$ and a central projection $c\in P\cap C(A)$ such that 
$spcs\leq qc$ and $sqc\sp{\perp}s\leq pc\sp{\perp}$. Since 
$p\sim q$, there is a finite product of symmetries $x$ such that 
$xpx\sp{\ast}=q$. Thus, $spcs\leq qc=xpcx\sp{\ast}$, whence 
$e:=x\sp{\ast}spcsx\leq pc$ with $sxex\sp{\ast}s=pc$. Therefore, 
$e\leq pc$ with $e\sim pc$, and since $pc\in M$, $e=pc$ by (iv), 
and it follows that $spcs=xex\sp{\ast}=xpcx\sp{\ast}=qc$.  
Likewise, $f:=xsqc\sp{\perp}sx\sp{\ast}\leq xpc\sp{\perp}x\sp
{\ast}=qc\sp{\perp}$ with $sx\sp{\ast}fxs=qc\sp{\perp}$, and 
we deduce that $f=qc\sp{\perp}$, whence $sqc\sp{\perp}s=x\sp{\ast}fx=
x\sp{\ast}qc\sp{\perp}x=pc\sp{\perp}$, so $spc\sp{\perp}s=qc\sp{\perp}$. 
Consequently, $sps=spcs+spc\sp{\perp}s=qc+qc\sp{\perp}=q$. 
Conversely, if $p$ and $q$ are exchanged by a symmetry, then 
$p\sim q$.

(vi) By hypothesis, there are finite products of symmetries 
$u$ and $x$ such that $q\sb{1}:=upu\sp{\ast}\leq q$ and $p\sb{1}:=
xqx\sp{\ast}\leq p$. Thus, $xq\sb{1}x\sp{\ast}\leq xqx\sp{\ast}
=p\sb{1}\leq p$ with $xq\sb{1}x\sp{\ast}=xupu\sp{\ast}x\sp{\ast}
\sim p$. By (iv), $xq\sb{1}x\sp{\ast}=p$, and therefore $q\sb{1}
=x\sp{\ast}px$. Consequently, $q=x\sp{\ast}p\sb{1}x\leq x\sp{\ast}
px=q\sb{1}$, so $q\sb{1}=q$, whence $p\sim q$. 
\end{proof}

Examination of the results in \cite{Top65} required for Topping's 
proof of his version of the type-I/II/III decomposition theorem 
for a JW-algebra \cite[Theorem 13]{Top65} now shows that all of 
these results either have been obtained above (often assuming 
that $P$ is complete) or follow easily from the results above. 
Therefore, we claim that our first project has been accomplished. 
We now focus on our second project. 

\section{The fundamental direct-decomposition\newline theorem} 
\label{sc:fundamental} 

\emph{The assumption that $P$ is centrally orthocomplete is 
still in force.}

\begin{assumption} In this section and the next, we assume that 
$Q$ is a TD subset of $P$.
\end{assumption}

We note that our subsequent results, apart from Theorem \ref
{th:type I/II/IIIprops}, do not require completeness of the OML 
$P$, nor do they require that $Q$ is STD. The terminology in 
the following definition is borrowed from \cite[pp. 28--29]{Top65}.

\begin{definition} \label{df:typeQetc}
Let $c\in P\cap C(A)$. Then:
\begin{enumerate}
\item [(1)] $c$ is \emph{type-$Q$} iff $c\in Q$.
\item [(2)] $c$ is \emph{locally type-Q} iff $c\in\gamma(Q)$. 
\item [(3)] $c$ is \emph{purely non-$Q$} iff no nonzero 
 subprojection of $c$ belongs to $Q$.
\item [(4)] $c$ is \emph{properly non-$Q$} iff no nonzero 
 central subprojection of $c$ belongs \linebreak to $Q$.
\end{enumerate}
\end{definition}

If $c\in P\cap C(A)$ and if $c$ is type-$Q$ (respectively, 
locally type-$Q$, purely non-$Q$, etc.), one also says that 
the direct summand $P[0,c]$ of $P$ and the direct summand 
$cA$ of $A$ are type-$Q$ (respectively, locally type-$Q$, 
purely non-$Q$, etc.).

We note that, by Theorem \ref{th:BMTSTD} (iv), for central 
projections $c\in P\cap C(A)$, the notion of local 
modularity introduced in Definition \ref{df:abelianetc} (3) 
is consistent with Definition \ref{df:typeQetc} (2), i.e., 
$c\in M\sb{0}$ iff $c$ is locally type-$M$.

\begin{theorem} [{\cite[Theorem 5.2]{FPType}}] \label{th:typeprops}
Let $c\in P\cap C(A)$. Then{\rm:}
\begin{enumerate}
\item $c$ is type-$Q$ iff $c\in Q\cap\gamma(Q)=Q\cap C(A)$ iff 
 every central subprojection of $c$ belongs to $Q\cap C(A)$ iff 
 $c\leq c\sb{Q\cap C(A)}$.
\item If $Q$ is STD, then $c$ is type-$Q$ iff $P[0,c]\subseteq Q$. 
\item $c$ is locally type-$Q$ iff every central subprojection of 
 $c$ belongs to $\gamma(Q)$ iff $c\leq c\sb{Q}$. 
\item $c$ is purely non-$Q$ iff $Q\cap P[0,c]=\{0\}$ iff $c\leq 
 (c\sb{Q})\sp{\perp}$.
\item $c$ is properly non-$Q$ iff the only central projection 
 in $Q\cap P[0,c]$ is $0$ iff $c\leq(c\sb{Q\cap C(A)})
 \sp{\perp}$.
\end{enumerate}
\end{theorem}

\begin{corollary}
Let $c,d\in P\cap C(A)$. Then{\rm: (i)} If $c$ is type-$Q$, then 
$c$ is locally type-$Q$. {\rm(ii)} If $c$ is purely non-$Q$, then 
$c$ is properly non-$Q$. {\rm(iii)} If $c$ is both type-$Q$ and 
properly non-$Q$, then $c=0$. {\rm(iv)} If $c$ is both locally 
type-$Q$ and purely non-$Q$, then $c=0$. {\rm(v)} If $c$ is 
type-$Q$ {\rm(}respectively, locally type-$Q$, purely non-$Q$, 
properly non-$Q${\rm)}, then so is $c\wedge d$. {\rm(vi)} If both 
$c$ and $d$ are type-$Q$ {\rm(}respectively, locally type-$Q$, purely 
non-$Q$, properly non-$Q${\rm)}, then so is $c\vee d$. 
\end{corollary}

\begin{lemma} [{\cite[Lemma 5.5]{FPType}}] \label{lm:binarydecomp}
{\rm(i)} There exists a unique central projection $c$, namely 
$c=c\sb{Q}$, such that $A=cA\oplus c\sp{\perp}A$, $cA$ is locally 
type-$Q$, and $c\sp{\perp}A$ is purely non-$Q$; moreover, $Q
\subseteq P[0,c\sb{Q}]$. {\rm(ii)} There exists a unique central 
projection $d$, namely $d=c\sb{Q\cap C(A)}$, such that $A=dA
\oplus d\sp{\perp}A$, $dA$ is type-$Q$, and $d\sp{\perp}A$ is 
properly non-$Q$; moreover, $Q\cap C(A)\subseteq P[0,c
\sb{Q\cap C(A)}]$.
\end{lemma}

The following theorem results from combining the direct decompositions 
in parts (i) and (ii) of Lemma \ref{lm:binarydecomp}. We regard 
this theorem as the \emph{fundamental direct-decomposition 
theorem} for the synaptic algebra $A$.

\begin{theorem} [{\cite[Theorem 5.6]{FPType}}]
Corresponding to the TD set $Q$, there exist unique pairwise 
orthogonal central projections $c\sb{1}$, $c\sb{2}$ and 
$c\sb{3}$, namely $c\sb{1}=c\sb{Q\cap C(A)}$, $c\sb{2}=c
\sb{Q}\wedge(c\sb{Q\cap C(A)})\sp{\perp}$, and $c\sb{3}=
(c\sb{Q})\sp{\perp}$, such that $c\sb{1}+c\sb{2}+c\sb{3}=1$;
\[
A=c\sb{1}A\oplus c\sb{2}A\oplus c\sb{3}A;
\] 
$c\sb{1}A$ is type-$Q$; $c\sb{2}A$ is locally type-$Q$, but properly 
non-$Q$; and $c\sb{3}$ is purely non-$Q$. Moreover, $Q\cap C(A)=
(P\cap[C(A))[0,c\sb{1}]$, $Q\subseteq P[0,c\sb{1}+c\sb{2}]$, and 
$(P\cap C(A))[0,c\sb{2}+c\sb{3}]\cap Q=\{0\}$.  
\end{theorem}

\section{The type-I/II/III decomposition theorem} \label{sc:I/II/III}

\emph{The assumption that $P$ is centrally orthocomplete is still 
in force.}

\begin{assumption}
In this section, we continue to assume that $Q\subseteq P$ is TD, 
and we also assume that $K\subseteq P$ is TD and that $Q\subseteq K$. 
\end{assumption}

\noindent Since $Q\subseteq K$, we have $c\sb{Q}\leq c\sb{K}$ and $c
\sb{Q\cap C(A)}\leq c\sb{K\cap\gamma(K)}$.

\begin{definition}
Let $c\in P\cap C(A)$. Then, with respecct to the pair of TD sets 
$Q\subseteq K$:
\begin{enumerate}
\item [(1)] $c$ is \emph{type} I iff it is locally type-$Q$.
\item [(2)] $c$ is \emph{type} II iff it is locally type $K$, but 
 purely non-$Q$.
\item [(3)] $c$ is \emph{type} III iff it is purely non-$K$. 
\item [(4)] $c$ is \emph{type} I$\sb{K}$ (respectively, \emph{type} 
 II$\sb{K}$) iff it is type I (respectively, type II) and also 
 type-$K$.
\item [(5)] $c$ is \emph{type} I$\sb{{\widetilde K}}$ (respectively, 
 \emph{type} II$\sb{{\widetilde K}}$) iff it is type I (respectively, 
 type II and also properly non-$K$).
\end{enumerate} 
\end{definition}

\noindent If $c\in P\cap C(A)$ and if $c$ is type I (respectively, 
type II, type III, etc.),  one also says that the direct  
summand $P[0,c]$ of $P$ and the direct summand $cA$ of 
$A$ are type I (respectively, type II, type III, etc.).

\begin{lemma} \label{lm:alttypeIconds}
Let $c\in P\cap C(A)$. Then the following conditions are 
mutually equivalent{\rm: (i)} $c$ is type I. {\rm(ii)} 
There is a projection $q\in Q$ such that $\gamma q=c$.
{\rm(iii)} There is a projection $q\in Q\cap P[0,c]$ 
that is faithful in $P[0,c]$. {\rm(iv)} Every nonzero direct summand 
of $P[0,c]$ contains a nonzero projection in $Q$. {\rm(v)} 
$c\leq c\sb{Q}$.  
\end{lemma}

\begin{proof}
(i) $\Leftrightarrow$ (ii) is the definition of $c$ being locally 
type-$Q$, (ii) $\Leftrightarrow$ (iii) follows from Lemma 
\ref{lm:relfaith}, and (i) $\Leftrightarrow$ (iv) $\Leftrightarrow$ 
(v) follows from Theorem \ref{th:typeprops} (iii).  
\end{proof}

The following is the \emph{type}-I/II/III \emph{decomposition theorem 
for synaptic algebras}. It is obtained by combining the fundamental 
direct-decomposition theorems for $Q$ and for $K$.  

\begin{theorem} [{\cite[Theorem 6.4]{FPType}}] \label{th:I/II/III}
Corresponding to the pair of TD sets $Q$ and $K$ with $Q\subseteq K$, 
there are unique pairwise orthogonal central projections $c\,\sb{\rm I}$, 
$c\,\sb{\rm II}$ and $c\,\sb{\rm III}$, namely $c\,\sb{\rm I}=c\sb{Q}$, 
$c\,\sb{\rm II}=c\sb{K}\wedge(c\sb{Q})\sp{\perp}$, and $c\,\sb{\rm III}
=(c\sb{K})\sp{\perp}$, such that $c\,\sb{\rm I}+c\,\sb{\rm II}+c\,\sb
{\rm III}=1$; 
\[
A=c\,\sb{\rm I}A\oplus c\,\sb{\rm II}A\oplus c\,\sb{\rm III}A;
\]
and $c\,\sb{\rm I}A$, $c\sb{\rm II}A$, and $c\,\sb{\rm III}A$ are 
of types {\rm I}, {\rm II}, and {\rm III}, respectively. Moreover, 
there are further decompositions 
\[
c\,\sb{\rm I}A=c\,\sb{\rm I\sb{K}}A\oplus c\,\sb{\rm I
 \sb{{\widetilde K}}}A\text{\ \ and\ \ }c\,\sb{\rm II}A=c
 \,\sb{\rm II\sb{K}}A\oplus c\,\sb{\rm II\sb{{\widetilde K}}}A,
\]
where  $c\,\sb{\rm I\sb{K}}$, $c\,\sb{\rm I\sb{{\widetilde K}}}$, 
$c\sb{\rm II\sb{K}}$, and $c\,\sb{\rm II\sb{{\widetilde K}}}$ are 
central projections of types {\rm I}$\sb{K}$, {\rm I}$
\sb{{\widetilde K}}$, {\rm II}$\sb{K}$, and {\rm II}$
\sb{{\widetilde K}}$, respectively; these decompositions are 
also unique; and 
\[
c\,\sb{\rm I\sb{K}}=c\sb{Q}\wedge c\sb{K\cap\gamma(K)},\   
 c\,\sb{\rm I\sb{{\widetilde K}}}=c\sb{Q}\wedge
 (c\sb{K\cap\gamma(K)})\sp{\perp}, 
\]
\[
c\sb{\rm II\sb{K}}=c\sb{K\cap\gamma(K)}\wedge(c\sb{Q})
\sp{\perp},\ c\,\sb{\rm II\sb{{\widetilde K}}}
=c\sb{K}\wedge(c\sb{K\cap\gamma(K)})\sp{\perp}\wedge(c\sb{Q})
\sp{\perp}.
\]
Furthermore, the type {\rm I}$\sb{K}$ direct summand decomposes 
as
\[
c\,\sb{\rm I\sb{K}}A=c\sb{11}A\oplus c\sb{21}A, 
\]
where $c\sb{11}$ and $c\sb{21}$ are central projections, $c\sb{11}$ 
is type-$Q$ {\rm(}hence also of type-$K${\rm)}, and $c\sb{21}$ is 
type-$K$, locally type-$Q$, but properly non-$Q$. The latter 
decomposition is also unique, and 
\[  
c\sb{11}=c\sb{Q\cap C(A)},\ \  c\sb{21}=c\sb{K\cap\gamma(K)}
 \wedge c\sb{Q}\wedge(c\sb{Q\cap C(A)})\sp{\perp}.
\]
\end{theorem}

\begin{theorem} \label{th:type I/II/IIIprops}
With the notation of Theorem \ref{th:I/II/III}{\rm:}
\begin{enumerate}
\item $Q\subseteq P[0,c\,\sb{{\rm I}}]$ and $K\subseteq P[0,
 c\,\sb{{\rm I}}+c\,\sb{{\rm II}}]$.
\item If $P$ is complete, $Q$ is projective, and $Q$ is STD, 
 then $c\,\sb{{\rm I}}=\bigvee Q$ and $Q$ is orthodense in 
 $P[0,c\,\sb{{\rm I}}]$.
\item If $P$ is complete, $K$ is projective, and $K$ is STD, 
 then $c\,\sb{{\rm I}}+c\,\sb{{\rm II}}=\bigvee K$ and $K$ 
 is orthodense in $P[0,c\,\sb{{\rm I}}+c\,\sb{{\rm II}}]$.
\end{enumerate}
\end{theorem}

\begin{proof}
(i) As $c\,\sb{{\rm I}}=c\sb{Q}$, and $c\,\sb{{\rm I}}+
c\,\sb{{\rm II}}=c\sb{Q}+(c\sb{K}-c\sb{Q})=c\sb{K}$, (i) follows 
from Lemma \ref{lm:typecover} (iv). In view of part (i), parts 
(ii) and (iii) follow from Corollary \ref{co:orthodense}.
\end{proof}

In Theorem \ref{th:I/II/III}, the unique five-fold direct-sum  
decomposition 
\[
A=c\,\sb{\rm I\sb{K}}A\oplus c\,\sb{\rm I\sb{{\widetilde K}}}A
 \oplus c\,\sb{\rm II\sb{K}}A\oplus c\,\sb
 {\rm II\sb{{\widetilde K}}}A\oplus c\,\sb{\rm III}A
\]
of $A$ into direct summands of types {\rm I}$\sb{K}$, {\rm I}$
\sb{{\widetilde K}}$, {\rm II}$\sb{K}$, {\rm II}$\sb
{{\widetilde K}}$ and {\rm III} is a generalization of the 
classic type-I/II/III decomposition for a von Neumann algebra 
(see Remark \ref{rm:vonNeumannalg} below); moreover,the additional 
decomposition $c\,\sb{\rm I\sb{K}}A=c\sb{11}A\oplus c\sb{21}A$ 
into direct summands of type-$Q$ and of type-$K$, locally type-$Q$, 
but properly non-$Q$ yields a six-fold direct decomposition of $A$,
\[
A=c\sb{11}A\oplus c\sb{21}A\oplus c\,\sb{\rm I\sb{{\widetilde K}}}A
 \oplus c\,\sb{\rm II\sb{K}}A\oplus c\,\sb{\rm II\sb{{\widetilde K}}}
 A\oplus c\,\sb{\rm III}A.
\]
Of course, if $A$ is a factor, then it is of precisely one of 
these six types.   

\begin{remark} \label{rm:vonNeumannalg}
If $R$ is a von Neumann algebra and $A$ is the synaptic algebra 
of all self-adjoint elements of $A$, then one obtains the classic 
type-I/II/III decomposition of $A$ (and also of $R$) by taking 
$Q=B$, the TD set of abelian projections in $A$, and taking $K$ 
to be the set of all finite projections in $A$.
\end{remark}

\begin{remark}
If $A$ is a JW-algebra, regarded as a synaptic algebra, then one 
obtains Topping's version of a type-I/II/III decomposition 
\cite[Theorem 13]{Top65} by taking $Q=B$ and $K=M$. 
\end{remark}


\begin{thebibliography}{99}

\bibitem{Alf} Alfsen, E.M., \emph{Compact Convex Sets and Boundary
Integrals}, Springer-Verlag, New York, 1971, ISBN 0-387-05090-6.

\bibitem{Beran} Beran, L., {\em Orthomodular Lattices, An Algebraic 
Approach}, Mathematics and its Applications, Vol. 18, D. Reidel Publishing 
Company, Dordrecht, 1985.

\bibitem{Chev} Chevalier, G., Around the relative center property in 
orthomodular lattices, \emph{Proc. Amer. Math. Soc.} {\bf 112} (1991), 
935--948.

\bibitem{OMLNote} Foulis, D.J., A note on orthomodular lattices, 
\emph{Portugal. Math.} {\bf 21} (1962) 65--72.

\bibitem{FSynap} Foulis, D.J., Synaptic algebras, \emph{Math. Slovaca} 
{\bf 60}, no. 5 (2010) 631--654.

\bibitem{FBeffect} Foulis, D.J. and Bennett, M.K., Effect algebras 
and unsharp quantum logics, \emph{Found. Physics} {\bf 24}, no. 10 
(1994) 1331--1352.

\bibitem{GHAlg1} Foulis, D.J. and Pulmannov\'{a}, S., Generalized
Hermitian Algebras, \emph{Int. J. Theor. Phys.} {\bf 48}, no. 5 (2009) 
1320--1333.

\bibitem{FPSynap} Foulis, D.J. and Pulmannov\'{a}, S., Projections in 
synaptic algebras, \emph{Order} {\bf 27}, no. 2 (2010) 235--257.

\bibitem{COEA} Foulis, D.J. and Pulmannov\'{a}, S., Centrally 
orthocomplete effect algebras, \emph{Algebra Univers.} {\bf 64} 
(2010) 283--307, DOI 10.007/s00012-010-0100-5.

\bibitem{FPType} Foulis, D.J. and Pulmannov\'{a}, S., Type-decomposition 
of an effect algebra, \emph{Found. Physics} {\bf 40} (2010) 1543--1565, 
DOI 10.1007/s10701-009-9344-3.

\bibitem{GHAlg2} Foulis, D.J. and Pulmannov\'{a}, S., Regular elements
in generalized Hermitian Algebras, \emph{Math. Slovaca} {\bf 61}, no. 2 
(2011) 155--172.

\bibitem{HandD} Foulis, D.J. and Pulmannov\'{a}, S., Hull mappings and 
dimension effect algebras, \emph{Math. Slovaca} {\bf 61}, no. 3 (2011) 
1--38.

\bibitem{SymSA} Foulis, D.J., and Pulmannov\'{a}, S., Symmetries in 
synaptic algebras, ArXiv:1304.4378. 

\bibitem{GPBB} Gudder, S., Pulmannov\'{a}, S., Bugajski, S., and
Beltrametti, E., Convex and linear effect algebras, \emph{Rep. Math.
Phys.} {\bf 44}, no. 3 (1999) 359--379.

\bibitem{Haag} Haag, R., \emph{Local Quantum Physics, Fields, 
Particles, Algebras}, Springer, Berlin, 1992.

\bibitem{SSH64} Holland, S.S, Jr., Distributivity and perspectivity 
in orthomodular lattices, \emph{Trans. Amer. Math. Soc.} {\bf 112} 
(1964) 330-343.

\bibitem{SSH70} Holland, S.S., Jr., An $m$-orthocomplete orthomodular
lattice is $m$-complete, \emph{Proc. Amer. Math. Soc.} {\bf 24} (1970)
716--718.

\bibitem{JPorthocompl} Jen\v{c}a, G. and Pulmannov\'{a}, S., 
Orthocomplete effect algebras, \emph{Proc. Amer. Math. Soc.} 
{\bf 131} (2003) 2663--2671.

\bibitem{Kalm} Kalmbach, G., \emph{Orthomodular Lattices}, Academic Press,
London, New York, 1983.

\bibitem{Kapcg} Kaplansky, I., Any orthocomplemented complete modular 
lattice is a continuous geometry, \emph{Ann. of Math.} {\bf 61} 
(1955) 524--541.

\bibitem{McC} McCrimmon, K.,  \emph{A taste of Jordan algebras},
Universitext, Springer-Verlag, New York, 2004, ISBN: 0-387-95447-3.

\bibitem{vNqm} Neumann, J. von, \emph{Mathematische Grundlagen der 
Quantenmechanik} Springer, Heidelberg, 1932.

\bibitem{vNcg} Neumann, J. von, \emph{Continuous geometry}, Princeton Univ. 
Press, Princeton 1960.

\bibitem{PuNote} Pulmannov\'{a}, S., A note on ideals in synaptic algebras, 
\emph{Math. Slovaca} {\bf 62}, no. 6 (2012) 1091-–1104.

\bibitem{Redei} R\'{e}dei, Mikl\'{o}s, Why John von Neumann did not like 
the Hilbert space formalism of quantum mechanics (and what he liked 
instead), \emph{Stud. Hist. Phil. Mod. Physics} {\bf 27}, no. 4 (1996) 
493--510.

\bibitem{ZRlatEA} Rie\v{c}anov\'{a}, Z., Subdirect decompositions of lattice 
effect algebras, \emph{Internat. J. Theoret. Phys.} {\bf 42}, no. 7 (2003)
1425–-1433. 

\bibitem{Top65} Topping, D.M.,  \emph{Jordan Algebras of Self-Adjoint 
Operators}, A.M.S. Memoir {\bf No 53} AMS, Providence, Rhode Island, 1965.

\end{thebibliography}
\end{document}